\documentclass{amsart}

\usepackage[numbers]{natbib}
\usepackage{amsmath, amssymb,stmaryrd}% ,mathdots?
\usepackage{type1cm}
\usepackage{tensor}
\usepackage{enumerate}
\usepackage{graphicx}
\usepackage[czech,english]{babel}
\usepackage{mdwtab}
\usepackage[OT1]{fontenc}
\usepackage[all]{xy}
\newcommand{\sig}{\lambda}

\newcommand{\entails}{\vdash}

\newcommand{\bbbn}{\mathbb{N}}

\newcommand{\bbbr}{\mathbb{R}}
\newcommand{\bbbc}{\mathbb{C}}

\usepackage{color}
\definecolor{shadecolor}{gray}{.85}%
\definecolor{tintedcolor}{gray}{.80}%
\usepackage{framed}%
\definecolor{mytintedcolor}{gray}{.95}%
\makeatletter
\newdimen\svparindent
\newcounter{tmpthm}
\newcounter{tmpconj}
\newenvironment{mainconjecture}[1]{\setcounter{tmpconj}{\value{conjecture}}%
\setcounter{conjecture}{#1}\addtocounter{conjecture}{-1}%
\begin{conjecture}}{\setcounter{conjecture}{\value{tmpconj}}\end{conjecture}}
\newenvironment{mytinted}{%
  \MakeFramed {\FrameRestore}}%
{\endMakeFramed}
{\endlist\end{mytinted}\egroup}
\makeatother

\newtheorem{theorem}{Theorem}
\newtheorem{proposition}{Proposition}
\newtheorem{corollary}{Corollary}
\newtheorem{lemma}[theorem]{Lemma}

\theoremstyle{definition}
\newtheorem{definition}[theorem]{Definition}
\newtheorem{example}[theorem]{Example}

\theoremstyle{remark}

\newtheorem{fact}[theorem]{Fact}
\newtheorem{problem}{Problem}
\newtheorem{conjecture}{Conjecture}
\begin{document}

\title{First-Order Limits, an Analytical Perspective}
\author{Jaroslav Ne{\v s}et{\v r}il}
\address{Jaroslav Ne{\v s}et{\v r}il\\
Computer Science Institute of Charles University (IUUK and ITI)\\
   Malostransk\' e n\' am.25, 11800 Praha 1, Czech Republic}
\email{nesetril@iuuk.mff.cuni.cz}
\thanks{Supported by grant ERCCZ LL-1201 
and CE-ITI P202/12/G061, and by the European Associated Laboratory ``Structures in
Combinatorics'' (LEA STRUCO)}

\author{Patrice Ossona~de~Mendez}
\address{Patrice~Ossona~de~Mendez\\
Centre d'Analyse et de Math\'ematiques Sociales (CNRS, UMR 8557)\\
  190-198 avenue de France, 75013 Paris, France
	--- and ---
Computer Science Institute of Charles University (IUUK)\\
   Malostransk\' e n\' am.25, 11800 Praha 1, Czech Republic}
\email{pom@ehess.fr}
\thanks{Supported by grant ERCCZ LL-1201 and by the European Associated Laboratory ``Structures in
Combinatorics'' (LEA STRUCO)}

\date{\today}



\begin{abstract}
In this paper we present a novel approach to graph (and structural) limits based on model theory and analysis. 
The role of Stone and Gelfand dualities is displayed prominently and leads to a general theory, which we believe is naturally emerging.
This approach covers all the particular examples of structural convergence and it put the whole in new context. As an application, it leads to new intermediate examples of structural convergence and to a ``grand conjecture'' dealing with sparse graphs. We survey the recent developments.
\end{abstract}

\maketitle

\section{Introduction}
\label{sec:intro}
Let $G_1,G_2,\dots,G_n$ be a sequence of graphs with increasing orders. The theory of graph limits allows us to capture some common features of such a sequence by a single structure, the limit object. Such features include property testing and subgraph frequencies.
Positive instances of this procedure are related to both dense and sparse graphs.
This recently very active area of combinatorics is discussed for instance is \cite{LovaszBook}.

Can one demand that the limit object satisfies some further properties such as degree distribution, or bounded diameter, or non-existence of certain subgraphs?

This we approach by the  notion of first-order limits of structures, which was introduced in~\cite{CMUC,NPOM1arxiv}. 

This approach is based on the combination of model theory and functional analysis, and proved to be useful in dealing with ``intermediate'' classes, such as the class of trees and forests \cite{NPOM2arxiv}.
It also generalizes (and as well puts in a new context) several notions of graph convergence studied earlier \cite{Borgs2005,Lov'asz2006,
Benjamini2001,Aldous2006}.

In Sections~\ref{sec:mu}--\ref{sec:cons}
we review here some of these results. However we also provide a new setting which generalizes and complements the earlier work of the authors \cite{CMUC}. We now view our approach as
a natural approach, which lies midway between functional analysis approach to quantum theory and limits of graphs. This, perhaps even closer to functional analysis, we tried to explain in Section~\ref{sec:ana}.

Here is an outline of what we do:
We generalize aspects of the theory of graph limits, moving from a study a homomorphism and subgraph profiles to considering full statistics of first-order formula satisfaction. 
In our setting the underlying notion of convergence is, in essence, model theoretic, a relies on the the following notion~\cite{CMUC}:

For a $\sigma$-structure $\mathbf A$ and a first-order formula $\phi$ (in the language of $\sigma$, with free variables $x_1,\dots,x_p$),  we denote by $\phi(\mathbf A)$ the {\em satisfying set} of $\phi$ in $\mathbf A$:
$$
\phi(\mathbf A)=\{(v_1,\dots,v_p)\in A^p:\ \mathbf{A}\models\phi(v_1,\dots,v_p)\},
$$
and we define the  {\em Stone pairing} of $\phi$ and $\mathbf A$ as 
the probability 
\begin{equation}
\label{eq:sp}
\langle\phi,\mathbf{A}\rangle=\frac{|\phi(\mathbf{A})|}{|A|^p}
\end{equation}
that $\mathbf{A}$ satisfies $\phi$ for a (uniform independent) random interpretation of the random variables.

A sequence $(\mathbf{A}_n)_{n\in\bbbn}$ of finite $\sigma$-structures
 is {\em FO-convergent} if
 the sequence $(\langle\phi,\mathbf{A}_n\rangle)_{n\in\bbbn}$ converges for every first-order formula $\phi$.
It is important that one can derive a weakened
notion of {\em $X$-convergence} (for a fragment $X$ of first-order logic) by restricting the range of the test formulas $\phi$
to $X$.

We shall see that this abstract generalization of left convergence and local convergence can be seen as a particular (commutative) case of a general framework, which occurs naturally in functional analysis (in its description of quantum physics). 
This connection not only legitimates some of the constructions introduced to study first-order limits of structures, but it also put the whole framework in perspective, for a possible extension to the non-commutative setting of quantum logic. This connection will be discussed in Section~\ref{sec:ana}, which may be seen as generalization and more uniform treatment of \cite{NPOM1arxiv}. 

Encouraged by this general analytic setting, we show  that first-order limits (shortly ${\rm FO}$-limits) and, more generally, $X$-limits can be uniquely represented by a probability measure $\mu$ on the Stone space $S$ dual to the Lindenbaum-Tarski algebra of the  formulas. In this setting, there is a one-to-one map 
$\mathbf A\mapsto \mu_{\mathbf A}$ from the 
class of finite $\sigma$-structures to the space of probability measures on $S$, with
 \begin{equation}
 \label{eq:mu0}
 \int_S \mathbf 1_\phi(T)\,{\rm d}\mu_{\mathbf A}(T)= \langle\phi,\mathbf{A}\rangle,
 \end{equation}
where $\mathbf 1_\phi$ is the indicator function of the clopen subset of $S$ dual to the formula $\phi$ in Stone duality.
In this setting, a sequence $(\mathbf{A}_n)_{n\in\bbbn}$ of finite
$\sigma$-structures is first-order convergent if and only if the measures $\mu_{\mathbf{A}_n}$ converge weakly to some measure $\mu$. 
 In such a case, the probability measure $\mu$ represents the limit of the sequence  $(\mathbf{A}_n)_{n\in\bbbn}$ and, for every first-order formula $\phi$ it holds
 \begin{equation}
 \label{eq:mu}
 \int_S \mathbf 1_\phi(T)\,{\rm d}\mu(T)=
 \lim_{n\rightarrow\infty}\int_S \mathbf 1_\phi(T)\,{\rm d}\mu_{\mathbf{A}_n}(T)=
 \lim_{n\rightarrow\infty} \langle\phi,\mathbf{A}_n\rangle.
 \end{equation}
 
This answers in a very general form the problem (originally posed by J. Chayes, see \cite[Introduction]{LovaszBook}) whether convergence leads to a limit distribution.

The particular definition~\eqref{eq:sp} of the Stone pairing of a first-order formula and a finite structure can be extended, under certain conditions, to infinite structures. An obvious necessary condition for such an extension is that the domain of the considered infinite structure should be a measurable space, with the property that every first-order definable set is measurable. Note that the difficulty here stands in the fact that the class of definable sets is closed by projections, what is not the case of measurable sets. 
Consequently, we introduce in Section~\ref{sec:mod} the following notions: a {\em Relational Sample Space} (RSS) is a relational structure, whose domain is a standard Borel space, with the property that every first-order definable subset of a power of the domain is measurable (with respect to Borel product measure). A {\em modeling} $\mathbf M$  \cite{NPOM1arxiv} is an RSS equipped with a probability measure denoted $\nu_{\mathbf M}$. In this setting, we can extend the definition of Stone pairing to modelings: for a modeling $\mathbf{M}$ and a first-order formula $\phi$ with free variables $x_1,\dots,x_p$ we define
\begin{equation}
\label{eq:spm}
\langle\phi,\mathbf{M}\rangle=\idotsint \mathbf 1_{\phi(\mathbf M)}(v_1,\dots,v_p)\,{\rm d}\nu_{\mathbf{M}}(v_1)\,\dots\,{\rm d}\nu_{\mathbf{M}}(v_p),
\end{equation}
where $\mathbf 1_{\phi(\mathbf M)}$ is the indicator function of the set
\begin{equation}
\label{eq:sol}
\phi(\mathbf M)=\{(v_1,\dots,v_p)\in M^p:\ \mathbf M\models\phi(v_1,\dots,v_p)\}.
\end{equation}

This definition naturally leads to the two following representation problems, which can be seen as  analogs of celebrated Aldous-Lyons problem on graphings. 

\begin{problem}
Characterize those probability measures $\mu$ on $X$ for which there exists
a sequence $(\mathbf A_n)_{n\in\bbbn}$ of finite structures such that for every first-order formula $\phi(x_1,\dots,x_p)$ it holds
\begin{equation}
\label{eq:goodmu}
\int_S \mathbf 1_\phi(T)\,{\rm d}\mu(T)=\lim_{n\rightarrow\infty}\langle\phi,\mathbf{A}_n\rangle.
\end{equation}
\end{problem}

\begin{problem}
Characterize those first-order convergent sequences of finite structures $(\mathbf{A}_n)_{n\in\bbbn}$ for which there exists
a modeling $\mathbf M$ such that for every first-order formula $\phi$ it holds
\begin{equation}
\label{eq:modlim}
\langle\phi,\mathbf{M}\rangle=\lim_{n\rightarrow\infty}\langle\phi,\mathbf{A}_n\rangle.
\end{equation}
\end{problem}

As we shall see, such modeling limits do not exist in general (but sometimes they do) and thus we shall be interested in the following problem.
\begin{problem}
\label{pb:modlim}
Characterize classes with modeling limits, that is classes $\mathcal C$  of finite structures, such that for every first-order convergent sequence $(\mathbf A_n)_{n\in\bbbn}$ of structures in $\mathcal C$  there is a modeling $\mathbf{M}$ such that $\langle\phi,\mathbf{M}\rangle=\lim_{n\rightarrow\infty}\langle\phi,\mathbf{A}_n\rangle$ for every first-order formula $\phi$.
\end{problem}
It appears that this last problem is, perhaps surprisingly, related to the sparse--dense dichotomy.
The authors introduced in \cite{ND_characterization,ND_logic} a dichotomy of classes of graphs,
between (sparse) {\em nowhere dense} classes and (dense) {\em somewhere dense} classes. This dichotomy, which can be expressed in numerous (and non trivially equivalent) ways, appears to be 
deeply related to first-order related properties of these classes.
Based on a characterization of nowhere dense classes by means of Vapnik-Chervonenkis dimension
of model-theoretic interpretations~\cite{Adler2013} and a characterization of random free
hereditary classes of graphs~\cite{Lovasz2010}, we proved in~\cite{NPOM1arxiv} that a monotone class
of graphs with modeling limits is necessarily nowhere dense.
We believe that nowhere dense--somewhere dense dichotomy actually gives the answer to Problem~\ref{pb:modlim}.
\begin{conjecture}[\cite{NPOM1arxiv}]
\label{conj:ND}
Every monotone nowhere dense class of graphs has modeling limits.
\end{conjecture}

By a combination of model theory and functional analysis methods,
we have been able to prove some particular cases: classes of graphs with bounded degree, classes of structures with bounded tree-depth~\cite{NPOM1arxiv}, as well as the class of all forests~\cite{NPOM2arxiv} admit modeling limits. In this paper, we sketch a generalization  to any class of graphs where the maximum degree is bounded by a function of the girth (or unbounded if the graph is acyclic).
But before, in the next section, we outline an even more general approach (pointing to a possible non-commutative version).
\section{Abstract Analytic Framework}
\label{sec:ana}
We introduced in~\cite{CMUC} 
the notion of first-order limits of structures 
as a generalization and unified treatment of various
 notions of limits. This notions also provided limit objects for new ``intermediate classes''. However 
 first-order limits are interesting on their own. Particularly, as we show here, they display a rich spectrum of interconnection to
the functional
analysis approach to quantum theory~\cite{Engesser2009}.
For instance, the connection to  $C^*$-algebras makes more natural the introduction of several spaces and constructions, whose appearance in our framework may have earlier seem a little mysterious.

The formalism we adopt here is indeed close to the one of quantum theory, specially when considering significant objects and spaces.
This aspect will be developed in this section, where we consider our approach to structure limits under the lights of statistical and quantum physics, as formalized by functional analysis.

Let us recall some basics:
The set of all possible states of a physical system is the {\em phase space} of the system. In statistical physics, a {\em statistical ensemble} is the
  phase space  of a physical system together with a method of averaging physical quantities, called {\em observables}, related to this system.
In a classical system with phase space $\Omega$, the observable quantities are real functions defined on $\Omega$, and they are averaged by integration with respect to a certain probability measure $\mu$ on $\Omega$. In a quantum system described by vectors in a Hilbert space $\mathcal H$, the observable quantities are defined by self-adjoint operators acting on $\mathcal H$, and are averaged using a certain positive, normalized functional $\rho$, defined on the algebra $\mathfrak A(\mathcal H)$ of operators on  $\mathcal H$ (such functionals on $\mathcal H$ are called {\em states}) \cite{Minlos2002}.

\subsection{$C^*$-algebras}
A new unifying view consists in considering as a primary object the $C^*$-algebra $\mathfrak A$ of observables. Recall that a $C^*$-algebra $\mathfrak A$ is a Banach algebra with an involution $x\mapsto x^*$, such that the relation $\|x^*x\|=\|x\|^2$ holds for any element $x\in\mathfrak A$.
The two standard examples of $C^*$-algebras are the following (see \cite{Gelfand1943} for the standard text):

First example is the space $C_0(X)$ of all continuous complex-valued functions which vanish at infinity on a locally compact Hausdorff space $X$, equipped with the uniform norm $\|f\|=\sup_{x\in X}|f(x)|$, and involution defined as the complex-conjugate: $f^*(x)=\overline{f(x)}$. Gelfand's representation theorem states that every commutative $C^*$-algebra $\mathfrak A$ is isometrically $*$-isomorphic to the algebra $C_0(\Phi_{\mathfrak{A}})$, where $\Phi_{\mathfrak{A}}$ is the topological space of all the {\em characters} of $\mathfrak{A}$ (that is: of the non-zero homomorphisms $f:\mathfrak{A}\rightarrow\bbbc$) equipped with the relative weak-$*$ topology.
Given $a\in \mathfrak A$, one defines the {\em Gelfand transform} $\widehat{a}$ of $a$, that is the
function $\widehat{a}:\Phi_{\mathfrak A}\to{\mathbb C}$ by $\widehat{a}(f)=f(a)$. Then the map $a\mapsto \widehat{a}$ defines a norm-decreasing, unit-preserving algebra homomorphism from $\mathfrak A$ to $C_0(\Phi_{\mathfrak{A}})$, which is the {\em Gelfand representation} of $\mathfrak A$.

The second standard example is the algebra $\mathfrak B(\mathcal H)$ of all bounded linear operators on a Hilbert space $\mathcal H$, where involution is defined as the adjoint operator, and norm is defined as the operator norm. Gelfand--Naimark theorem states that
any $C^*$-algebra is isometrically and symmetrically isomorphic to a $C^*$-subalgebra of some $C^*$-algebra of the form  $\mathfrak B(\mathcal H)$ .

In our (discrete) setting, we shall be  interested in the specific case where the algebra $\mathfrak A$ of observables is {\em approximately finite} \cite{Bratteli1972}, that is in the case where there exists an increasing sequence
$\mathfrak A_1\subseteq \mathfrak A_2\subseteq \dots\subseteq \mathfrak A$ of finite-dimensional sub-$C^*$-algebras of $\mathfrak A$ such that the union $\bigcup_{j}\mathfrak A_j$ is norm-dense in $\mathfrak A$. Approximately finite $C^*$-algebras (or  {\em AF $C^*$-algebras}) have been fully classified by Elliott \cite{Elliott197629} using $K$-theory for $C^*$-algebras.
In particular, a $C^*$-algebra $\mathfrak A$ is  commutative and approximately finite if and only if it is the algebra of continuous functions on the Stone space of a Boolean algebra (follows from \cite[Proposition 3.1]{Bratteli1974192} and Stone duality). 
This very particular case corresponds to the framework we used to study structural limits. 
Fig.~\ref{fig:std} describes the different mathematical objects entering this grand picture.

\begin{figure}[h!bt]
\begin{center}
\includegraphics[width=\textwidth]{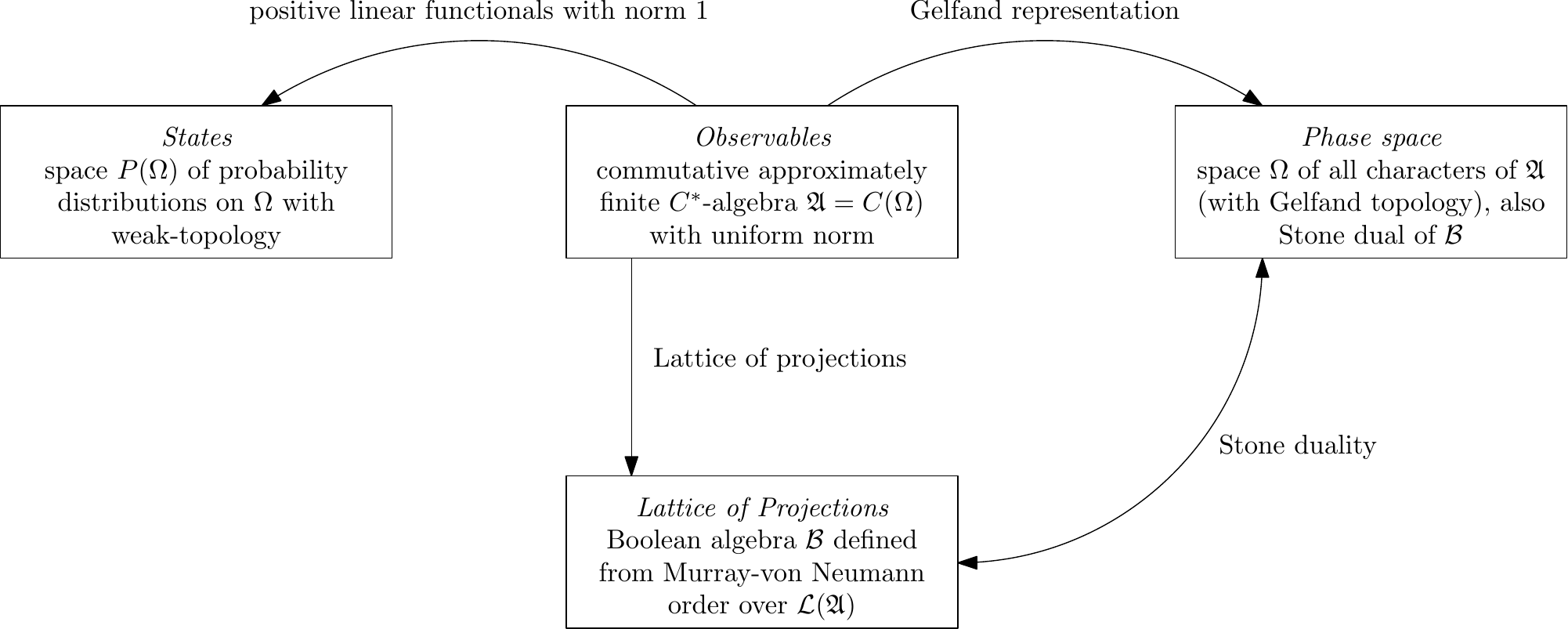} 
\end{center}
\caption{The different spaces in the standard case}
\label{fig:std}
\end{figure}

In this context, let us
recall that to each Boolean algebra 
$\mathcal B$ is associated a topological space, denoted  $S(\mathcal B)$, called the {\em Stone space} of $\mathcal B$, which is a totally disconnected compact Hausdorff space. 
The points in $S(\mathcal B)$ are the ultrafilters on $\mathcal B$, or equivalently the homomorphisms from $\mathcal B$ to the two-element Boolean algebra. 
The mapping $K:\mathcal{B}\rightarrow 2^{S(\mathcal{B})}$ defined by
$$K(\phi)=\{T\in S(\mathcal{B}: \phi\in T\}$$
gives a one-to-one correspondence between elements of $\mathcal{B}$ and  clopen subsets of $S(\mathcal{B})$.
The topology on $S(\mathcal B)$ is generated by a basis consisting of all (clopen) sets of the form
$\{ x \in S(B) \mid b \in x\}$, where $b$ is an element of $B$.
For every Boolean algebra $\mathcal B$, $S(\mathcal B)$ is a compact totally disconnected Hausdorff space, and
 Stone representation theorem \cite{Stone1936} states that every 
Boolean algebra $\mathcal B$ is isomorphic to the algebra of clopen subsets of its Stone space $S(\mathcal B)$. 
 Note that the Stone dual of a countable Boolean algebra is a compact Polish space.
The similarity of Stone duality with Gelfand representation is not surprising, if one considers the motivation that led Stone to this theorem:
\begin{quote}
``The writer's interest in the subject, for example, arose in connection with the spectral theory of symmetric transformations in Hilbert space and certain related properties of abstract integrals.'' (M.H. Stone \cite{Stone1936})
\end{quote}

\subsection{Structural Limits}
In the context of structural limits, the general picture
of Fig.~\ref{fig:std} takes more concrete form depicted on 
Fig.~\ref{fig:FO}. This is a refinement of the approach the authors presented in \cite{CMUC,NPOM1arxiv,NPOM2arxiv}.

\begin{figure}[htb]
\begin{center}
\includegraphics[width=\textwidth]{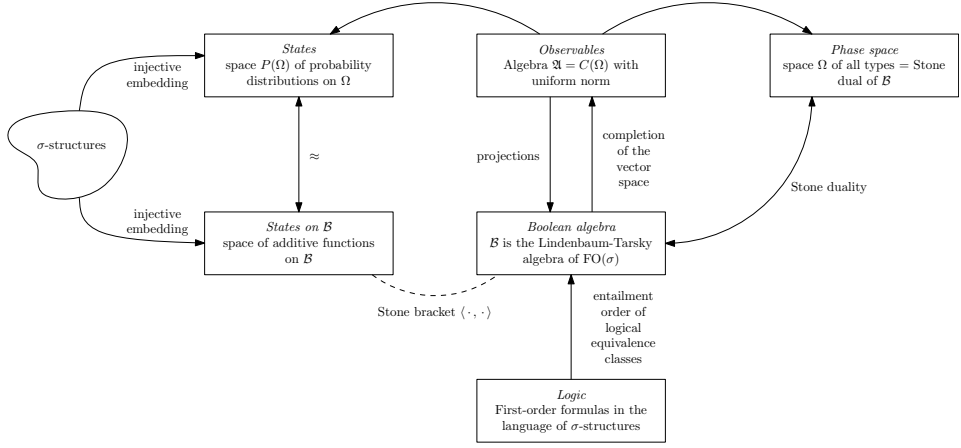} 
\end{center}
\caption{The spaces considered in the study of structural limits}
\label{fig:FO}
\end{figure}

Let us explain the role of the different spaces entering this scheme:
We consider the class ${\rm FO}(\sigma)$ of all first-order formulas constructed with equality and the symbols in the considered signature $\sigma$, and the
 Boolean algebra $\mathcal B$ defined as the Lindenbaum-Tarski algebra of ${\rm FO}(\sigma)$, that is the Boolean algebra of the  classes of these formulas with respect to logical equivalence.
 For instance, in the case of undirected graphs, formulas are constructed from the symbols of equality and adjacency, as well as standard connectives (conjunctions, disjunction, negation) and quantification over vertices.

The Stone space $S(\mathcal B)$ associated to $\mathcal B$ is a topological space, whose topology is generated by clopen subsets, which correspond exactly to first-order formulas (up to logical equivalence), that is to elements of $\mathcal B$. A point in $S(\mathcal B)$ corresponds to a maximal set of consistent formulas.

As to each formula $\phi$ is associated a clopen subset $K(\phi)$ of $S(\mathcal B)$, the indicator functions $\mathbf 1_{K(\phi)}$ of clopen subsets of  
$S(\mathcal B)$ are continuous functions from $S(\mathcal B)$ to $\mathbb R$.
An algebra $\mathfrak A$ can be constructed from $\mathcal B$ as the algebra
of formal finite real linear combinations of elements of $\mathcal B$ quotiented
by the relation $\phi+\psi=\phi\wedge\psi+\phi\vee\psi$. Each such finite linear combination can be written as a sum $x=\sum_{i=1}^n \alpha_i\phi_i$ where
$\phi_i\wedge\phi_j=0$ whenever $i\neq j$. It turns out that $\|x\|=\max_i |\alpha_i|$ is independent of the chosen decomposition, and defines a norm on $\mathfrak A$ (this norm is nothing but the $\|\,\|_\infty$ norm on $C(S(\mathcal B))$). 
The completion of $\mathfrak A$ for this norm is the Banach algebra $C(S(\mathcal B))$ of all continuous functions from $S(\mathcal{B})$ to $\mathbb{R}$.
(Note that if we consider for $\mathfrak A$ complex linear combinations, we similarly construct the $C^*$-algebra of all complex valued functions on $S(\mathcal B)$.)

Finite $\sigma$-structures embed injectively in the space  of additive functions
on the Boolean algebra $\mathcal B$ by $\mathbf A\mapsto\langle\,\cdot\,,\mathbf A\rangle$, where 
$$
\langle\phi,\mathbf A\rangle=\frac{|\{(v_1,\dots,v_p)\in A^p:\ \mathbf A\models\phi(v_1,\dots,v_p)\}|}{|A|^p}.
$$
The following fact is worth noticing:
\begin{fact} 
There is a one-to-one correspondence between
probability measures $\mu$ on $S(\mathcal{B})$ and additive functions 
$f_\mu:\mathcal{B}\rightarrow[0,1]$ such that $f(1)=1$, with the property that for every probability measure $\mu$ on $S(\mathcal{B})$ and every $\phi\in\mathcal{B}$ it holds
$$f_\mu(\phi)=\mu(K(\phi)).$$
\end{fact}
This fact, which we proved in \cite{CMUC} as a lemma, is actually known
in the more general setting of MV-algebras \cite{Kroupa2006,Panti2008}.
We give here a short proof for the sake of completeness:
\begin{proof}
The vector space $V(\mathcal{B})$ generated by the indicator functions $\mathbf 1_{K(\phi)}$ for $\phi\in\mathcal{B}$  forms a subalgebra of $C(S(\mathcal{B}))$. The subalgebra $V(\mathcal{B})$ separates points of $S(\mathcal{B})$: if $T_1,T_2\in S(\mathcal{B})$,
then there exists $\phi\in T_1\setminus T_2$ thus 
$\mathbf 1_{K(\phi)}(T_1)\neq \mathbf 1_{K(\phi)}(T_2)$.
As the constant $1$ function (corresponding to the maximum of the Boolean algebra) belongs to $V(\mathcal{B})$, the subalgebra $V(\mathcal{B})$ is dense in $C(S(\mathcal{B}))$, according to Stone--Weierstrass theorem. 
As $S(\mathcal{B})$ is a Radon-space, probability measures $\mu$ on $S(\mathcal{B})$ correspond
(in a one-to-one correspondence) to positive functionals $F_\mu$ on $C(S(\mathcal{B}))$ such that  $F_\mu(1)=1$ (according to Riesz representation theorem). As $V(\mathcal{B})$ is dense in $C(S(\mathcal{B}))$, every positive functional on $C(S((\mathcal{B}))$ is uniquely determined by its restriction on $V(\mathcal{B})$, and conversely every positive functional on $V(\mathcal{B})$ uniquely extend (by continuity) to a positive functional on 
$C(S((\mathcal{B}))$. Moreover, a norm $1$ positive functional on $V(\mathcal{B})$ is uniquely determined by its values on the generating set $\{\mathbf 1_{K(\phi)}: \phi\in\mathcal{B}\}$, that is by the function $f_\mu:\mathcal{B}\rightarrow [0,1]$
defined by $f_\mu(\phi)=\mu(K(\phi))$. It is easily checked that $f_\mu$ satisfies $f_\mu(1)=1$ and
$$\phi\wedge\psi=0 \quad\Longrightarrow\quad f_\mu(\phi\vee\psi)=f_\mu(\phi)+f_\mu(\psi),$$
and that every function $f$ satisfying these inequalities uniquely defines a norm $1$ positive functional on $V(\mathcal{B})$, hence a measure $\mu$ on $S(\mathcal{B})$, such that $f=f_\mu$.
\end{proof}

It follows that finite $\sigma$-structures embed injectively in the space 
$P(S(\mathcal{B}))$ of probability distributions on $S(\mathcal{B})$: to a structure $\mathbf A$ with domain $A$ we associate the unique measure $\mu_{\mathbf A}$ such that for every $\phi\in{\rm FO}(\sigma)$ with free variables $x_1,\dots,x_p$ it holds
$$
\int_{S(\mathcal B)}\mathbf 1_{K(\phi)}\,{\rm d}\mu_{\mathbf A}^{\otimes p} =%\mu_{\mathbf A}^{\otimes p}(K(\phi))=
\langle\phi,\mathbf A\rangle.
$$
Then, a sequence $(\mathbf A_n)_{n\in\bbbn}$ is ${\rm FO}$-convergent
if and only if the probability measures $\mu_{\mathbf{A}_n}$ converge weakly.

More can be said when considering convergence of finite $\sigma$-structures.
The group $S_\omega$ of permutations of $\bbbn$ acts naturally on ${\rm FO}(\sigma)$ by permuting the free variables: for a formula $\phi$ with free variables $x_{i_1},\dots,x_{i_p}$ and for a permutation $\tau$, the formula $\tau\cdot\phi$ with free variables $x_{j_1},\dots,x_{j_p}$
where $j_k=\tau(i_k)$ is defined by
 $$
 (\tau\cdot\phi)(x_{j_1},\dots,x_{j_p}):=
 \phi(x_{i_1},\dots,x_{i_p}).
 $$
 It follows that $S_\omega$ acts on the Boolean algebra $\mathcal{B}$, and the
 action $S_\omega\curvearrowright\mathcal B$ defines an action
  $S_\omega\curvearrowright S(\mathcal B)$ on the Stone space of $\mathcal B$. As it is clear that for every $\sigma$-structure $\mathbf A$, every permutation $\tau$, and every first-order formula $\phi$ it holds
$$
\langle\tau\cdot\phi,\mathbf A\rangle=\langle\phi,\mathbf A\rangle.
$$
Hence if a sequence $(\mathbf A_n)_{n\in\bbbn}$ is ${\rm FO}$-convergent, then the measures $\mu_{\mathbf A_n}$ converge weakly to some
$S_\omega$-invariant  probability measure on $S(\mathcal{B})$.
More generally, every injection $f:\bbbn\rightarrow \bbbn$ defines a transformation $P_f:S(\mathcal B)\rightarrow S(\mathcal B)$, which is continuous, which is  a homeomorphism if (and only if) $f$ is bijective (that is a permutation on $\bbbn$). The set of the transformations $P_f$ has a monoid structure, as
$P_f\circ P_g=P_{f\circ g}$. Every measure $\mu$ that is obtained either from a finite structure or as weak limit of measures associated to finite structures is invariant under every $P_f$. Hence $(S(\mathcal B),\Sigma_{S(\mathcal B)},\mu,P)$ is a measure preserving dynamical system, where $\Sigma_{S(\mathcal B)}$ is the Borel $\sigma$-algebra of $S(\mathcal B)$ and $P$ is the monoid of transformations indexed by injections $f:\bbbn\rightarrow \bbbn$.

We summarize the situation as follows:
\begin{theorem}
\label{thm:1}
Let $\sigma$ be an at most countable signature, let $X\subseteq {\rm FO}(\sigma)$ be a fragment of first-order logic, let $\mathcal B$ be the Lindenbaum-Tarski algebra of $X$, let $S$ be the Stone dual of $\mathcal B$, and let $P(S)$ be the space of probability distributions on $S$. Then:

To each $\sigma$ structure $\mathbf A$ corresponds a measure
$\mu_{\mathbf{A}}$ such that for every formula $\phi\in X$ with free variables in $\{x_1,\dots,x_p\}$ it holds
$$\int_{S}\mathbf 1_{K(\phi)}\,{\rm d}\mu_{\mathbf{A}}
=\langle\phi,\mathbf{A}\rangle;$$

A sequence $(\mathbf{A}_n)_{n\in\bbbn}$ of $\sigma$-structures is $X$-convergent if and only if the associated measures $\mu_{\mathbf{A}_n}$ converge weakly to some measure $\mu\in P(S)$. In such a case, for every 
formula $\phi\in X$ with free variables in $\{x_1,\dots,x_p\}$ it holds
$$
\lim_{n\rightarrow\infty}\langle\phi,\mathbf{A}_n\rangle=\int_{S}\mathbf 1_{K(\phi)}\,{\rm d}\mu.
$$
Moreover, if $\Gamma$ is a group of automorphisms of $\mathcal B$ (for instance the group of permutations of the free variables), then $\Gamma$ acts naturally on $S$; if all the measures $\mu_{\mathbf{A}}$ associated to finite $\sigma$-structures are $\Gamma$-invariant, then so are all the measures obtained as weak limits of measures associated to finite $\sigma$-structures.
\end{theorem}

The proof is an easy modification of the proof given in \cite{CMUC}, where we proved Theorem~\ref{thm:1} without the appendix on group.
\subsection{Limits with respect to Parameters}
The above abstract setting can be combined with ideas underlying left limits of graphs (as explained in \cite{LovaszBook}) and, particularly, by property testing.
The universal framework, complementing the approach of \cite{CMUC}, and using some of the above theory, is to build a notion of limits as follows:

We consider  a set
 $\mathcal O$, whose elements we shall call
{\em objects}, and a set $\mathcal F$  of bounded mappings $f:\mathcal{O}\rightarrow\bbbr$, which we call {\em parameters}.
Then 
we define the notion convergence of a sequence of objects with respect to the family of parameters.

Because we are looking for a notion of limits with respect with the mappings in $\mathcal{F}$,
we consider on $\mathcal O$
the {\em initial topology} with respect to $\mathcal F$, that is the coarsest topology that makes every mapping $f\in\mathcal F$ continuous. Note that $\mathcal{O}$ is not Hausdorff in general.
Then the  space
$C_b(\mathcal O,\bbbr)$ of bounded continuous functions on
$\mathcal O$ with sup norm 
$\|f\|=\sup_{O\in\mathcal{O}}|f(O)|$
is a commutative Banach algebra (that can be extended to a commutative $C^*$-algebra).
According to Gelfand representation theorem, $C_b(\mathcal{O},\bbbr)$ is isometrically isomorphic to the algebra of continuous functions on a compact Hausdorff space, which is known to be homeomorphic to the Stone-\v Cech compactification $\beta\mathcal{O}$ of $\mathcal{O}$.
Recall that the
the Stone-\v Cech compactification $\beta\mathcal{O}$ of 
$\mathcal{O}$ is the largest compact Hausdorff space generated by $\mathcal{O}$, in the sense that any map from $\mathcal{O}$ to a compact Hausdorff space factors (in a unique way) through $\beta\mathcal{O}$.

Let ${\rm Profile}:\mathcal{O}\rightarrow [0,1]^{\mathcal{F}}$ denote the mapping defined by
$${\rm Profile}(O)=(f(O))_{f\in\mathcal F},$$
and let $\sim$ be the equivalence relation
$$O_1\sim O_2\quad\iff {\rm Profile}(O_1)={\rm Profile}(O_2).$$

Then $K\mathcal O=\mathcal O/\sim$ is the Kolmogorov quotient of $O$,
 and  $\beta\mathcal{O}$ is homeomorphic to the closure of $\{{\rm Profile}(O): O\in\mathcal{O}\}$ in
$[0,1]^\mathcal{F}$. 

This leads to the following natural notion of convergence of a sequence of objects of $\mathcal O$ with respect to parameters in $\mathcal{F}$:
\begin{definition} 
A sequence
$(O_n)_{n\in\bbbn}$ is {\em $\mathcal F$-convergent} if, for every
$f\in\mathcal F$ the sequence $(f(O_n))_{n\in\bbbn}$ is convergent.
\end{definition}
Note that the limit of an $\mathcal F$-convergent sequence of objects in $\mathcal O$ is uniquely defined as a point of $\beta\mathcal{O}$. 

If $P\in\mathbb{R}[x_1,\dots,x_p]$ is a real polynomial
and $f_1,\dots,f_p\in\mathcal F$, we denote by 
$P(f_1,\dots,f_p)$ the mapping from $\mathcal O$ to $\mathbb{R}$ defined by
$$
P(f_1,\dots,f_p): O\mapsto P(f_1(O),\dots,f_p(O)).
$$
Let $A(\mathcal{O},\mathcal{F})$ be the sub-algebra of $C_b(\mathcal O)$ formed by above mappings.
Stone-Weierstrass theorem on uniform approximation of continuous functions by polynomials extends in this setting (and this is a folklore):
\begin{lemma}
The sub-algebra $A(\mathcal{O},\mathcal{F})$  is dense in $C_b(\mathcal O)=C(\beta\mathcal{O})$.
\end{lemma}
\begin{proof}
The algebra $A(\mathcal{O},\mathcal{F})$ contains a non-zero constant (because of constant polynomials) and $A(\mathcal{O},\mathcal{F})$ separates the points of $\beta\mathcal{O}$ (as it separates the points of $K\mathcal{O}$).
Thus the result follows from Stone-Weierstrass theorem.
\end{proof}

Note that above results are very general, as we made no assumptions on our objects or parameters. Nevertheless, even that has some direct applications of this point of view to the theory of left-convergence.

Let ${\rm hom}(F,G)$ denote the number of homomorphisms from $F$ to $G$, and let $t(F,G)$ be the {\em homomorphism density} of $F$ is $G$ defined by
$$t(F,G)=\frac{{\rm hom}(F,G)}{|G|^{|F|}},$$
that is the probability that a random map from $F$ to $G$ is a homomorphism.

\begin{proposition}
Let $\mathcal O$ be the set of all finite graphs, and let $\mathcal F$ be the set of all the mappings $G\mapsto t(F,G)$ considered as mappings 
$$t(F,\,\cdot\,): \mathcal{O}\rightarrow [0,1].$$

Then the notion of convergence defined above is the 
left-convergence, and $\beta\mathcal{O}$ is homeomorphic to the space of graphons.
\end{proposition}

Let $\mathcal SQ$ denote the algebra of {\em standard quantum graphs}, that is the commutative algebra constructed from vector space with basis
formed by $\{K_1\}$ and all the finite graphs without isolated vertices,  by defining multiplication by linearity from the case of finite graphs, defining
 the unit of the algebra as $K_1$, and the 
product of two graphs without isolated vertices $F_1$ and $F_2$ as their disjoint union.

We shall make use of the following Lemma  \cite[Corollary 5.45, p.~74]{LovaszBook}:

\begin{lemma}
\label{lem:sqbasis}
The graph parameters ${\rm hom}(F,\,\cdot\,)$ (where $F$ ranges over simple graphs) are linearly independent. Equivalently, the  graph parameters ${\rm hom}(F,\,\cdot\,)$ (where $F$ ranges over connected simple graphs) are algebraically independent.
\end{lemma}

\begin{proposition}
Let $\mathcal O$ be the set of all finite graphs, and let $\mathcal F$ be the set of homomorphism densities
$$t(F,\,\cdot\,): \mathcal O\rightarrow [0,1].$$

Then the mapping
$$T:\mathcal{SQ}\rightarrow A(\mathcal{O},\mathcal{F})\qquad
\sum_{i=1}^\ell a_i\,F_i\mapsto \sum_{i=1}^\ell a_i t(F_i,\,\cdot\,)$$
is an algebra isomorphism of $\mathcal{SQ}$ and $A(\mathcal{O},\mathcal{F})$.
\end{proposition}
\begin{proof}
The algebraic independence of the parameters $t(F,\,\cdot\,)$ for graphs without isolated vertices will follow from
the linear independence of the parameters $t(F,\,\cdot\,)$ for simple graphs $F$ either equal to $K_1$ or without isolated vertices , which we now prove:

Assume that there exist non-isomorphic simple graphs $F_1,\dots,F_\ell$ (all of them being either $K_1$ or without isolated vertices) and reals $a_1,\dots,a_\ell$ such that
$$\sum_{i=1}^\ell a_i t(F_i,G)=0$$
holds  for every graph $G$.
Let $N=\max_{i=1}^\ell |F_i|$, and let $F_i'$ be the graph obtained from $F_i$ by adding $N-|F_i|$ isolated vertices. Then
$${\rm hom}(F_i',G)=|G|^{|F_i'|}t(F_i',G)=|G|^Nt(F_i,G).$$
Hence the equation $\sum_{i=1}^\ell a_i {\rm hom}(F_i',G)=0$ holds for every graph $G$. As no two graphs $F_i'$ are isomorphic, it follows from 
Lemma~\ref{lem:sqbasis} that $a_1,\dots,a_l$ are all zero.
\end{proof}
Let us give two direct consequences of this lemma.
The first consequence can be seen as an analog Stone Weierstrass  uniform approximation of a continuous function  by polynomials.

Say that a graph parameter $f$ is {\em left continuous} if the sequence $(f(G_n))_{n\in\bbbn}$ converges for every left convergent sequence of graphs $(G_n)_{n\in\bbbn}$.

\begin{proposition}
For every left-continuous graph parameter $f$ and every $\epsilon>0$ there exists
a quantum graph $\mathbf F=\sum_{i=1}^\ell a_i\,F_i$ such that for every graph $G$ it holds
$$
|t(\mathbf F,G)-f(G)|<\epsilon.
$$
\end{proposition}
\begin{proof}
This is an immediate consequence of the density of $T(\mathcal{SQ})$ in $C(\beta\mathcal{O})$.
\end{proof}

Say that a quantum graph $\mathbf{F}$ is {\em non-negative} (which we denote by $\mathbf{F}\geq 0$) if
$t(\mathbf{F},G)\geq 0$ for every graph $G$.
The non-negativity of quantum graphs was subject to intensive study
(see \cite{LovaszBook}). 
The second consequence gives yet another characterization of this property.

\begin{proposition}
\label{prop:sqrt}
Let $\mathbf F$ be a quantum graph. Then $\mathbf F\geq 0$ if and only if for every $\epsilon>0$ there exists a quantum graph $\mathbf H$ such that 
$\mathbf H\geq 0$ and $\|\mathbf F-\mathbf H^2\|<\epsilon$, that is:
$$
\inf_G t(\mathbf{H},G)\geq 0\quad\text{and}\quad
\sup_{G}|t(\mathbf F,G)-t(\mathbf H^2,G)|<\epsilon.$$
\end{proposition}
\begin{proof}
One direction is clear. For the other direction, assume $\mathbf F\geq 0$. Let $h$ be the
 left-continuous graph parameter defined by 
 $h(G)=(t(\mathbf F,G)+\epsilon/2)^{1/2}$.
Let $\mathbf H$ be a quantum graph such that $\|h-T(\mathbf H)\|<\alpha$, where
$\alpha(\|\mathbf F+\epsilon/2K_1\|^{1/2}+\alpha)<\epsilon/2$.
Then
$$\inf_G t(\mathbf H,G)\geq \inf_G t(h,G)-\alpha\geq (\epsilon/2)^{1/2}-\alpha\geq 0.$$
Moreover,
$$\|(\mathbf F+\epsilon/2K_1)-\mathbf H^2\|=\|h^2-T(\mathbf H^2)\|
\leq \|h+T(\mathbf H)\|\,\|h-T(\mathbf H)\|
\leq (2\|h\|+\alpha)\alpha<\epsilon/2.$$
Thus
$$
\|\mathbf F-\mathbf H^2\|\leq \|(\mathbf F+\epsilon/2K_1)-\mathbf H^2\|+\|\epsilon/2K_1\|<\epsilon.$$
\end{proof}

It is also pleasant to note that the sup norm defined on $C_b(\mathcal{O})$ defines naturally (by identification of $\mathcal{SQ}$ and $A(\mathcal{O},\mathcal{F})$) defines  a norm 
$$\Bigl\|\sum_i a_iF_i\bigr\|=\sup_{G\in\mathcal{O}}\,\Bigl|\sum_i a_i t(F_i,G)\Bigr|$$
for quantum graphs.

As we shall see later on, left-convergence of graphs is equivalent to (structural) ${\rm QF}^-$-convergence,
 where ${\rm QF}^-$ is the fragment of quantifier-free formula without equality. The global situation, which we just outlined, is depicted on Fig.~\ref{fig:QFq}
 \begin{figure}[ht]
 \includegraphics[width=\textwidth]{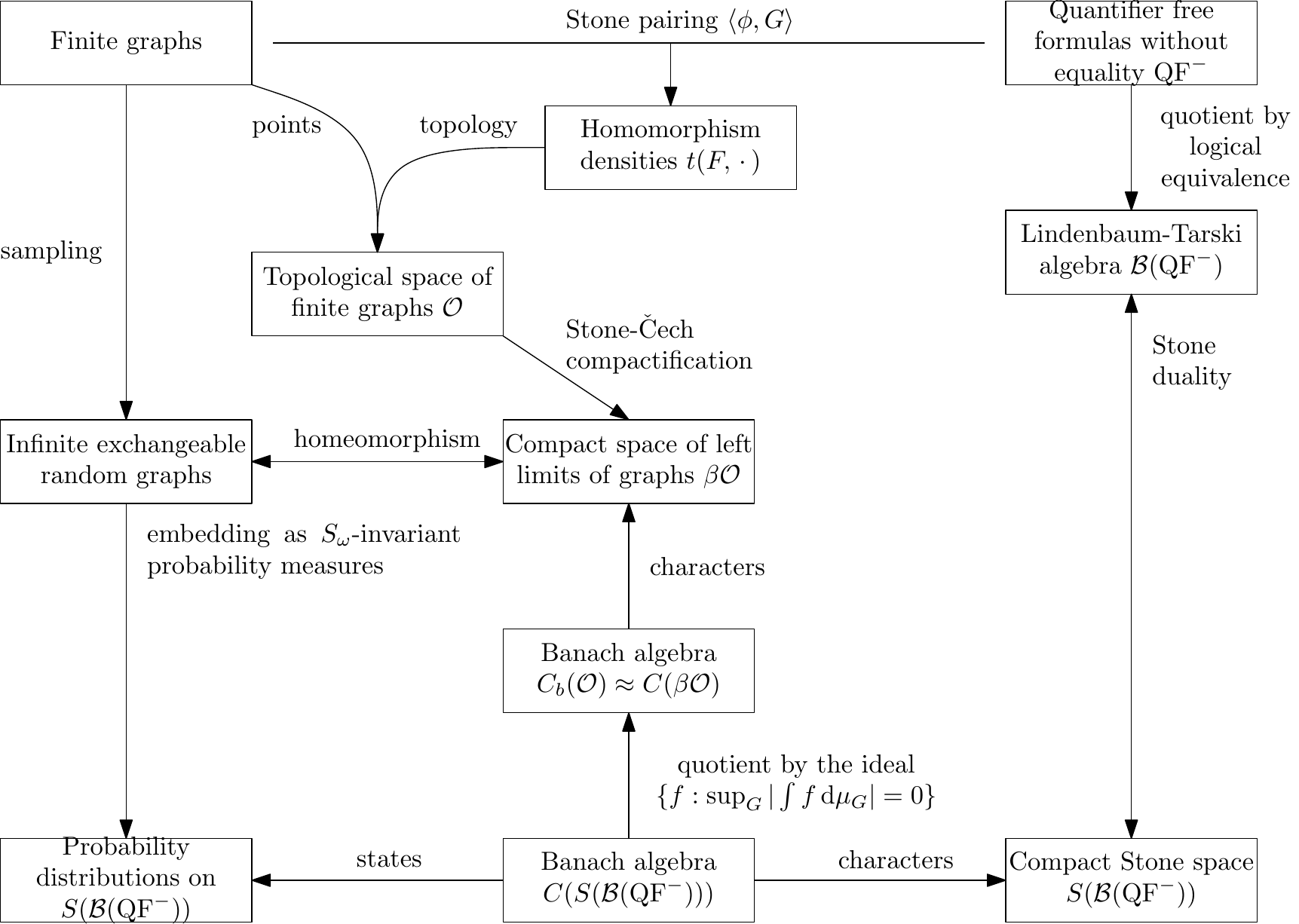}
\caption{Left-convergence and Structural ${\rm QF}^-$-convergence.
 \label{fig:QFq}}
 \end{figure}
\section{Fragments of Interest}
\label{sec:mu}

We review here a more detailed analysis of the convergence notions 
based of  Lindenbaum-Tarski algebras of different fragments, in the context of Section~\ref{sec:ana} and of our earlier approach \cite{CMUC}. We consider the following fragments:
\begin{itemize}
\item ${\rm FO}_0$: the fragment of sentences; 
\item ${\rm FO}_p$: the fragment of formulas with free variables included in $\{x_1,\dots,x_p\}$;
\item ${\rm QF}$: the fragment of quantifier-free formulas;
\item ${\rm FO}^{\rm local}$: the fragment of local formulas;
\item ${\rm FO}_1^{\rm local}$: the fragment of local formulas with single free variable.
\end{itemize}

Let us review some particular important cases of our theory induced by specific fragments of first-order logic.
\subsection{Elementary Limits}
Our starting case consists in considering {\em elementary convergence}:
A sequence $(\mathbf{A}_n)_{n\in\bbbn}$ is {\em elementary convergent} if, for every sentence (that is formulas without any free variables) $\theta$ there is an integer $N$ such that either all the $\mathbf{A}_n$ with $n\geq N$ satisfy $\theta$, or no $\mathbf{A}_n$ with $n\geq N$ satisfies $\theta$.
In other words, elementary convergence is convergence defined by the fragment ${\rm FO}_0(\sigma)$ of first-order sentences. The Stone dual $S$ of the Lindenbaum-Tarski algebra of first-order sentences is the space of  complete theories.
Recall that a {\em complete theory} is a maximal consistent set of sentences.
G\"odel's celebrated completeness theorem of first-order logic asserts that every complete theory theory has a model. In other words, for every consistent set
$T\subset {\rm FO}_0(\sigma)$ of sentences there exists a $\sigma$-structure
$\mathbf A$ that is a {\em model} of $T$, that is such that $\mathbf A$ satisfies every sentence in $T$. Moreover, according to L\"owenheim--Skolem theorem, if the signature $\sigma$ is at most countable then every consistent theory $T\subset {\rm FO}_0(\sigma)$ has a model whose domain is at most countable.

Conversely, to each $\sigma$-structure $\mathbf{A}$ one associates the
{\em complete theory}  ${\rm Th}(\mathbf{A})$ of $\mathbf{A}$, which is the set of all the sentences satisfied by $\mathbf A$.
Although it is easily checked to the mapping $\mathbf{A}\mapsto {\rm Th}(\mathbf{A})$ is $1-1$ on finite structures, this is not the case for infinite structures;  two structures $\mathbf{A}$ and $\mathbf{B}$ sharing the same complete theory are said to be {\em elementary equivalent}.

In such a context, the Stone bracket $\langle\phi,\mathbf A\rangle$ gets only values $0 $ (if $A\not\models\phi$) or $1$ (if $A\models\phi$) , and
the measure $\mu_{\mathbf{A}}$ associated to $\mathbf A$ is a Dirac measure at the point ${\rm Th}(\mathbf A)$. The class of all finite $\sigma$-structures can accordingly be identified with an open subset $O$ of $S$. The elementary limits of ${\rm FO}_0$-convergent sequences of finite $\sigma$-structures correspond  to the points of the closure of $O$, which correspond to the complete theories $T$ having the {\em finite model property}, that is the complete theories $T$ such a every finite subset of $T$ has a finite model.

It follows that every ${\rm FO}_0$-convergent sequence of finite
$\sigma$-structures has a limit that can be represented as a $\sigma$-structure with an at most countable domain. 

Note that  characterizing elementary limits of finite structures is essentially difficult, as finite model property is not decidable in general. However, it is worth noticing that  %Effectively Propositional Logic, that is the fragment of
first-order formulas without functional symbols where all existential quantifications appear first in the formula, has the finite model property.
Precisely,  Ramsey \cite{Ramsey1930}  showed that a sentence 
$$\exists x_1\dots\exists x_n\forall y_1\dots\forall y_m\ \phi$$
(where $\phi$ is quantifier-free and has no function symbols)
has a model if and only if it has a model of size bounded by $n$ plus the number of constants in $\phi$.
\subsection{Left limits and the Quantifier Free Fragment}
We can precise our intuition by
considering the case of undirected graphs with convergence driven by the fragment ${\rm QF}^-$ of quantifier free 
formulas without equality (meaning that the equality symbol cannot be used).
It is easily checked that the corresponding notion of convergence of a sequence
$(G_n)_{n\in\bbbn}$ of graphs is nothing but the {\em left-convergence} introduced by Lov\'asz {\em et al.}, which is based on the convergence of the {\em homomorphism profile} $t$, which associates to each finite graph $F$ the 
probability $t(F,G_n)$
that a random map from $F$ to $G_n$ is a homomorphism (that is an adjacency-preserving map). The Stone dual $S({\rm QF}^-)$ of the Lindenbaum-Tarski algebra of ${\rm QF}^-$ can be described as follows: each point $p\in S({\rm QF}^-)$ is a maximal consistent set of formulas  ${\rm QF}^-$, hence is uniquely determined by either $x_i\sim x_j$ or $\neg(x_i\sim x_j)$ 
 for each pair $(i,j)\in\bbbn^2$, stating that 
 $x_i$ is adjacent or not to $x_j$. Each point $p$ can thus be represented by means of a (labeled) countable graph with vertex set $\bbbn$, and the left limit
 of a sequence of graphs can be represented as an $S_\omega$-invariant  probability distribution on the Stone space, that is as an {\em infinite exchangeable random graph}. Then it follows from Aldous and Hoover extentions of de Finetti's theorem  to exchangeable arrays \cite{Aldous1985,Hoover1979} that left limits of graphs can be represented by means of a {\em graphon}, that is a symmetric measurable function
 $W:[0,1]\times [0,1]\rightarrow [0,1]$.
However, note that in the case of the fragment ${\rm QF}^-$, the mapping
$G\mapsto \mu_G$ is not injective, as two vertices that can be uniformly blown up in order to obtain a same graph are mapped to the same probability measure. This special aspect is due to the weakness of expressive power of the ${\rm QF}^-$ fragment, and it disappears when one considers the fragment ${\rm QF}$ of all quantifier-free formulas, which essentially defines the same notion of convergence.

\subsection{Local limits and the ${\rm FO}_1^{\rm local}$ fragment}
We now consider the case of undirected graphs with degree at most $D$, and the fragment  ${\rm FO}_1^{\rm local}$ of local formulas with a single free variable. Recall that a formula $\phi$ is {\em $t$-local} if its satisfaction only depends on a distance $t$-neighborhood 
of the free variables, and that $\phi$ is {\em local} if it is $t$-local for some $t\in\bbbn$. 

Local convergence of graphs with bounded degree has been defined by Benjamini and Schramm~\cite{Benjamini2001}. 
In our setting, local convergence can be defined as follows:
 A sequence $(G_n)_{n\in\bbbn}$ of graphs with maximum degree $D$ is {\em locally-convergent} if, for every $r\in\bbbn$, the distribution of the isomorphism types of the distance $r$-neighborhood of a random vertex of $G_n$ converges as $n$ goes to infinity. 
Local convergence witnesses the particular importance of the fragment ${\rm FO}_1^{\rm local}$:
 As each isomorphism type $T$ of the distance $r$-neighborhood of a vertex can be characterized by the satisfaction of an $r$-local formula formula $\phi_T(x)$ (where $x$ stands for the root vertex), it follows that ${\rm FO}_1^{\rm local}$-convergence implies local convergence. Conversely, as the satisfaction of an $r$-local formula with a single free variable only depends on the isomorphism type of the (rooted) distance $r$-neighborhood of the free variable,  ${\rm FO}_1^{\rm local}$-convergence is equivalent local convergence.

In general, points of the Stone dual of ${\rm FO}_1^{\rm local}$ correspond to the complete theories of rooted connected structures. As we consider graphs with bounded degrees, we know that two rooted connected countable graphs with maximum degree at most $D$ are elementarily equivalent if and only if they are isomorphic. It follows  that in this case the 
 Stone dual of ${\rm FO}_1^{\rm local}$ is the space of all (isomorphism classes of) rooted connected countable graphs with maximum degree $D$.

\subsection{Interrelations of Fragments and Reductions}

A peculiar feature of the fragments ${\rm FO}_0$ and ${\rm FO}^{\rm local}$, stated as  Theorem~\ref{thm:locel}, is a consequence of Gaifman locality theorem \cite{Gaifman1982}.  This we recall now:

\begin{theorem}
\label{thm:gaifman}
For every first-order formula $\phi(x_1,\dots,x_n)$ there exist integers
$t$ and $r$ such that $\phi$ is equivalent to a Boolean combination of $t$-local
formulas $\xi_s(x_{i_1},\dots,x_{i_s})$ and sentences of the form

\begin{equation}
\label{eq:gl}
\exists
y_1\dots\exists y_m\biggl(\bigwedge_{1\leq i<j\leq m}{\rm
dist}(y_i,y_j)>2r\wedge\bigwedge_{1\leq i\leq
m}\psi(y_i)\biggr)
\end{equation}

where $\psi$ is $r$-local. Furthermore, we can choose
 $$r\leq 7^{{\rm qrank}(\phi)-1},\ t\leq
(7^{{\rm qrank}(\phi)-1}-1)/2,\ m\leq n+{\rm qrank}(\phi),$$
 and, if
$\phi$ is a sentence, only sentences~\eqref{eq:gl} occur in the Boolean
combination. Moreover, these sentences can be chosen with quantifier rank
at most $f({\rm qrank}(\phi))$, for some fixed function $f$.  
\end{theorem}
The proof of the following reduction theorem essentially uses Theorem~\ref{thm:gaifman}.

\begin{theorem}[\cite{CMUC}]
\label{thm:locel}
Let $(\mathbf A_n)$ be a sequence of finite $\sig$-structures. Then $(\mathbf A_n)$ is ${\rm FO}$-convergent if
and only if it is both ${\rm FO}^{\rm local}$-convergent and ${\rm FO}_0$-convergent.
Precisely, $(\mathbf A_n)$ is ${\rm FO}_p$-convergent if
and only if it is both ${\rm FO}_p^{\rm local}$-convergent and ${\rm FO}_0$-convergent.
\end{theorem}

Of course, if we constrain the class of structures on which we study limits, we get stronger reductions.
For example, we now prove that in the case of a sequence of vertex-transitive structures in a wide class, most of the hierarchy of $X$-convergence collapses.
Recall that a class of graphs $\mathcal{C}$ is {\em wide} if,  for every
integers $d,m$ there is an integer $N(d,m)$ such that every graph $G\in\mathcal C$ with order at least $N(d,m)$ contains a subset of $m$ vertices pairwise at distance at least $d$ \cite{Atserias2005}. (In particular, every class of graphs with bounded degree is wide.)

\begin{lemma}
\label{lem:vtw}
For every sentence $\theta$ there exists constants $r,m$ and a formula $\phi\in{\rm FO}_1^{\rm local}$ such that for every vertex-transitive $\sig$-structure $\mathbf A$ with at least $m$ vertices pairwise at distance greater than $2r$ it holds
$$\mathbf A\models\theta\quad\iff\quad\mathbf A\models\exists x\,\phi(x)\quad\iff\quad\forall x\,\phi(x)\quad\iff\quad\langle\phi,\mathbf A\rangle=1.$$
\end{lemma}
\begin{proof}
If $\mathbf A$ is vertex-transitive then for every $u,v\in A$ and every $\phi\in {\rm FO}_1$ it holds $\mathbf{A}\models \phi(u)$ if and only if $\mathbf{A}\models \phi(v)$. Hence  $\mathbf{A}\models \exists x\,\phi(x)$ is equivalent to
$\mathbf{A}\models \forall x\,\phi(x)$.

According to Theorem~\ref{thm:gaifman}, the sentence $\theta$ is equivalent to a Boolean combination of sentences of the form
\begin{align*}
\quad&\mathbf{A}\models\exists
y_1\dots\exists y_m\biggl(\bigwedge_{1\leq i<j\leq m}{\rm
dist}(y_i,y_j)>2r\wedge\bigwedge_{1\leq i\leq
m}\psi(y_i)\biggr)\\
\iff &\mathbf{A}\models\exists
y_1\dots\exists y_m\bigl(\bigwedge_{1\leq i<j\leq m}{\rm
dist}(y_i,y_j)>2r\bigr)\wedge\exists y\bigwedge_{1\leq i\leq
m}\psi(y)\\
\iff&\mathbf{A}\models\exists y\bigwedge_{1\leq i\leq m}\psi(y)
\end{align*}
(where the first equivalence comes from vertex-transitivity, and the second from the existence of $m$ vertices far apart.)
Moreover, by vertex transitivity, for every $\phi_1,\phi_2\in{\rm FO}_1^{\rm local}$ it holds
$$\mathbf{A}\models (\exists x_1\, \phi_1(x_1))\wedge(\exists x_1\, \phi_2(x_2))
\quad\iff\quad\mathbf{A}\models (\exists x\, \phi_1(x)\wedge\phi_2(x)).
$$
It follows that the Boolean combination of the conditions 
$\mathbf{A}\models\exists y \bigwedge_{1\leq i\leq m}\psi(y)$ is equivalent to the condition $\mathbf{A}\models\exists x\,\phi(x)$ for some $\phi\in{\rm FO}_1^{\rm local}$.
\end{proof}

As a consequence, we obtain the following:

\begin{theorem}
\label{thm:vt}
Let $\mathcal C$ be a wide class of vertex-transitive $\sig$-structures. Then, for every sequence $(\mathbf A_n)_{n\in\bbbn}$ with $\mathbf A_n\in\mathcal{C}$ and $|A_n|\rightarrow\infty$ the following conditions are equivalent:
\begin{enumerate}
\item\label{it:1} the sequence $(\mathbf A_n)_{n\in\bbbn}$ is ${\rm FO}$-convergent,
\item\label{it:2} the sequence $(\mathbf A_n)_{n\in\bbbn}$ is ${\rm FO}_1^{\rm local}$-convergent,
\item\label{it:3} for every $\phi\in{\rm FO}_1^{\rm local}$, either  $(\exists x)\,\phi(x)$ is satisfied by all but finitely many $\mathbf{A}_n$ or  $(\exists x)\,\phi(x)$ is satisfied by only finitely many $\mathbf{A}_n$,
\item\label{it:4} the sequence $(\mathbf A_n)_{n\in\bbbn}$ is ${\rm FO}_0$-convergent.
\end{enumerate} 
\end{theorem}
\begin{proof}
That~\eqref{it:1} implies~\eqref{it:2} is trivial. That~\eqref{it:2} is equivalent to~\eqref{it:3} follows from vertex-transitivity. That~\eqref{it:3} implies~\eqref{it:4} follows from Lemma~\ref{lem:vtw} and that~\eqref{it:4} implies~\eqref{it:3} is trivial. That~\eqref{it:2} implies ${\rm FO}^{\rm local}$ convergence is a consequence of Gaifman locality theorem and wideness of $\mathcal C$. Finally, that~\eqref{it:4} together with ${\rm FO}^{\rm local}$ convergence implies~\eqref{it:1} follows from Theorem~\ref{thm:locel}.
\end{proof}

Let us formulate the following consequence for Cayley graphs.

\begin{corollary}
Let $d\in\bbbn$, and
let $(G_n)_{n\in\bbbn}$ be a sequence of (edge-colored directed) Cayley graphs generated by $d$ elements of groups with orders going to infinity. Then the following conditions are equivalent:
\begin{enumerate}
\item\label{itc:1} the sequence $(G_n)_{n\in\bbbn}$ is ${\rm FO}$-convergent,
\item\label{itc:2} the sequence $(G_n)_{n\in\bbbn}$ is ${\rm FO}_1^{\rm local}$-convergent,
\item\label{itc:4} the sequence $(G_n)_{n\in\bbbn}$ is ${\rm FO}_0$-convergent.
\end{enumerate} 
\end{corollary}
We have already noticed (Theorem~\ref{thm:locel}) that ${\rm FO}$-convergence can be reduced to  ${\rm FO}_0$-convergence and to
 ${\rm FO}^{\rm local}$-convergence (thanks to Gaifman locality theorem).
 Furthermore, for wide classes (like classes with bouded degree), 
 ${\rm FO}^{\rm local}$-convergence is equivalent to 
  ${\rm FO}_1^{\rm local}$-convergence. This explains the particular importance of the fragments ${\rm FO}^{\rm local}$ and ${\rm FO}_1^{\rm local}$.
  
\section{Limit Connectivity}
\label{sec:conn}
  The importance of the fragments reviewed in the previous section also appears in the study of the limit notion of ``connected components'' for a convergent sequence of $\sig$-structures. Note that connectivity of structures
  (which fails to be first-order definable) plays an important and non-trivial aspect of convergence of sparse graphs. The relative size and the speed of convergence plays a role here. Part of this problem is expressed by the following notions:

For a $\sigma$-structure $\mathbf A$, a vertex $v\in A$, and an integer $d$, we define $B_d(\mathbf A,v)$ as the substructure of $\mathbf A$ induced by vertices at distance at most $d$ from $v$ in $\mathbf A$ (that is at distance at most $d$ from $v$ in the Gaifman graph of $\mathbf A$).
 
 \begin{definition}[Residual sequence]
A sequence $(\mathbf A_n)_{n\in\bbbn}$ of $\sig$-structures is {\em residual} if
$$\lim_{d\rightarrow\infty} \limsup_{n\rightarrow\infty}\sup_{v\in A_n}\nu_{\mathbf A_n}(B_d(\mathbf A_n,v))=0.$$
\end{definition}
In other words,  a sequences is residual if its limit has only zero-measure connected components
\begin{definition}
A sequence $(\mathbf A_n)_{n\in\bbbn}$ of $\sig$-structures is {\em non-dispersive} if
$$\lim_{d\rightarrow\infty} \liminf_{n\rightarrow\infty}\sup_{v\in A_n} \nu_{\mathbf A_n}(B_d(\mathbf A_n,v))=1.$$
In other words,  a non-dispersive sequence is a sequence whose limit is essentially connected.

For a $\sigma$-structure $\mathbf A$ and a vertex $\rho\in A$, we denote by $(\mathbf A,\rho)$ the rooting of the structure $\mathbf A$ at $\rho$, that is the $\sigma^\bullet$-structure (where $\sigma^\bullet$ is obtained by adding a unary symbol $R$ to signature $\sigma$) obtained from $\mathbf A$ by putting exactly $\rho$ in relation $R$.

In the case of rooted structures, we usually want a stronger statement that the structures remain concentrated around their roots:
a sequence $(\mathbf A_n,\rho_n)_{n\in\bbbn}$ of rooted $\sig$-structures is {\em $\rho$-non-dispersive} if
$$\lim_{d\rightarrow\infty} \liminf_{n\rightarrow\infty}\nu_{\mathbf A_n}(B_d(\mathbf A_n,\rho_n))=1.$$
\end{definition}

The notion of non-dispersiveness is only an approximation of  connectiveness. As such it is not supported by a single equivalence relation. However, we can introduce the notion of a component relation system, which approximates component equivalence.

\begin{definition}
A {\em component relation system} for a class $\mathcal C$ of $\sig$-structures is a sequence
$\varpi_d$ of equivalence relations such that for every $d\in\bbbn$ and for 
every $\mathbf A\in\mathcal C$ there is a partition
of the $\varpi_d$-equivalence classes of $A$ into two parts $\mathcal E_0(\varpi_d,\mathbf A)$ and
$\mathcal E_+(\varpi_d,\mathbf A)$ such that:
\begin{itemize}
	\item every class in $\mathcal E_0(\varpi_d,\mathbf A)$ is a singleton;
	\item $\nu_{\mathbf A}(\bigcup\mathcal E_0(\varpi_d,\mathbf A))<\epsilon(d)+\eta(|A|)$ (where $\lim_{d\rightarrow\infty}\epsilon(d)=\lim_{n\rightarrow\infty}\eta(n)=0$);
	\item two vertices $x,y$ in $\bigcup\mathcal E_+(\varpi_d,\mathbf A)$ belong to a same connected component of $\mathbf A$ if and only if $\mathbf A\models\varpi_d(x,y)$ (i.e. iff $x$ and $y$ belong to a same class).
\end{itemize}
\end{definition}

\begin{definition}[\cite{NPOM1arxiv}]
A family of sequence $(\mathbf A_{i,n})_{n\in\bbbn}\ (i\in I)$ 
of $\sig$-structures is
{\em uniformly elementarily convergent} if, for every formula
$\phi\in{\rm FO}_1(\sig)$ there is an integer $N$ such that it holds
$$
\forall i\in I,\ \forall n'\geq n\geq N,\quad 
(\mathbf A_{i,n}\models (\exists x)\phi(x))\Longrightarrow
(\mathbf A_{i,n'}\models (\exists x)\phi(x)).
$$
\end{definition}
(Note that if a family $(\mathbf A_{i,n})_{n\in\bbbn}\ (i\in I)$
of sequences is uniformly elementarily convergent, then
each sequence $(\mathbf A_{i,n})_{n\in\bbbn}$ is elementarily convergent.)

Now that we have introduced all the necessary notions, we can state
the following fundamental result, which allows to study ${\rm FO}^{\rm local}$-convergent sequences (and more generally ${\rm FO}$-convergent sequences) of disconnected $\sig$-structures essentially component-wise.

\begin{theorem}[Extended comb structure \cite{NPOM2arxiv}]
\label{thm:ecomb}
Let $(\mathbf A_n)_{n\in\bbbn}$ be an ${\rm FO}^{\rm local}$-convergent sequence of
finite $\sig$-structures with component relation system $\varpi_d$.

Then there exist $I\subseteq\bbbn\cup\{0\}$ and, for each $i\in I$, a real $\alpha_i$ and a sequence
$(\mathbf B_{i,n})_{n\in\bbbn}$ of $\sig$-structures, such that
$\mathbf A_n=\bigcup_{i\in I}\mathbf B_{i,n}$ (for all $n\in\bbbn$), $\sum_{i\in{I}}\alpha_i=1$, and 
for each $i\in I$ it holds
\begin{itemize}
	\item $\alpha_i=\lim_{n\rightarrow\infty}\frac{|\mathbf B_{i,n}|}{|\mathbf A_n|}$, and
	$\alpha_i>0$ if $i\neq 0$;
	\item if $i=0$ and $\alpha_0>0$ then $(\mathbf B_{i,n})_{n\in\bbbn}$ is ${\rm FO}^{\rm local}$-convergent and
	residual;
\item if $i>0$ then $(\mathbf B_{i,n})_{n\in\bbbn}$ is ${\rm FO}^{\rm local}$-convergent
	 and non-dispersive.
	\end{itemize}
Moreover, if $(\mathbf A_n)_{n\in\bbbn}$ is ${\rm FO}$-convergent we can require
the family $\{(\mathbf B_{i,n})_{n\in\bbbn}: i\in I\}$ to be uniformly elementarily-convergent.
\end{theorem}

The proof of Theorem~\ref{thm:ecomb} is elaborated as it needs a detailed analysis  of the speed of growth of components and we refer to 
\cite{NPOM2arxiv} for proof.
An important consequence of this theorem is to reduce the problem of the construction of  limit objects to the basic cases of residual sequences and non-dispersive sequences, see for instance Theorem~\ref{thm:hmod}.

\section{Limit Objects of Sparse Structures}
\label{sec:mod}
\subsection{Borel Structures}
Given a signature $\sigma$, 
 {\em Borel $\sigma$-structure} is a $\sigma$-structure $\mathbf A$, whose domain $A$ is a standard Borel space, such that:
 \begin{itemize}
 \item for every relational symbol $R\in\sigma$ with arity $r$, the set
 $$R^{\mathbf A}=\{(v_1,\dots,v_r)\in A^r:\ \mathbf A\models R(v_1,\dots,v_r\}$$
of $R$ in $\mathbf A$ is a Borel subset of $A^r$;
 \item for every functional symbol $f\in\sigma$ with arity $r$, the corresponding map  $f^{\mathbf A}:A^r\rightarrow A$   is a Borel map.
 \end{itemize}
 Note that Borel structures are a natural generalization of Borel graphs (see \cite{Kechris1999}).
 
 It is easily checked that if $\mathbf A$ is a Borel $\sigma$-structure and if $\nu$ is a probability measure on $A$, then for every quantifier-free formula $\phi\in {\rm QF}(\sigma)$, the set
 $\phi(\mathbf A)$ is Borel. Whence we can extend
 the definition of Stone bracket to pairs formed by a Borel structure $\mathbf A$ equipped with a probability measure $\nu$ and a quantifier free formula $\phi$ with free variables in $x_1,\dots,x_p$ by
 $$
 \langle\phi,(\mathbf A,\nu)\rangle=\nu^{\otimes p}(\phi(\mathbf A)).$$
 Then, we shall say that a Borel $\sigma$-structure $\mathbf A$ equipped with a probability measure $\nu$ is a {\em QF-limit} of a 
 sequence $(\mathbf F_n)_{n\in\bbbn}$ of finite $\sigma$-structures if for every quantifier free formula $\phi$ it holds
 $$
 \langle\phi,(\mathbf A,\nu)\rangle=\lim_{n\rightarrow\infty}
  \langle\phi,\mathbf F_n\rangle.
 $$
 \begin{example}
 Let $(G_n)_{n\in\bbbn}$ be a left convergent sequence with
$|G_n|\rightarrow\infty$.

Then there exists a Borel graph $\mathbf G$ with a probability measure $\nu$ such that $(\mathbf G,\nu)$ is a QF-limit of the
sequence  $(G_n)_{n\in\bbbn}$ if and only if the sequence 
$(G_n)_{n\in\bbbn}$ converges to a {\em random-free} graphon, that is a graph which is $\{0,1\}$-valued almost everywhere.
 \end{example}
 
 Some more exotic examples can be obtained by considering functional symbols:
 \begin{example}[\cite{QFlim}]
 \label{ex:tsl}
 A {\em tree semi-lattice} is a structure $\mathbf A$ whose signature only contains a binary functional symbol $\wedge$, which satisfies the following axioms:
 \begin{enumerate}
 \item the operation (defined by) $\wedge$  is associative, commutative, and idempotent;
 \item for every $x,y,z$ such that
 $x\wedge z=x$ and  $y\wedge z=y$ it holds
 $x\wedge y\in\{x,y\}$.
 \end{enumerate}

Every QF-convergent sequence
$(\mathbf T_n)_{n\in\bbbn}$ of tree semi-lattices 
with $k$-colored domain there exists a Borel tree semilattice $\mathbf L$ with $k$-colored domain and a probability measure $\nu$ such that $(\mathbf L,\nu)$ is a QF-limit of the sequence $(\mathbf T_n)_{n\in\bbbn}$.
 \end{example}
 
 \begin{example}
 In this example,
let $(G_n)_{n\in\bbbn}$ be a sequence of finite nonabelian groups with increasing orders. Assume each $G_n$ is either simple or is a symmetric group.
Then the sequence  $(G_n)_{n\in\bbbn}$ is QF-convergent and
${\rm GL}(2,\bbbr)$ (with Haar probability measure) is a QF-limit
of $(G_n)_{n\in\bbbn}$.

Indeed,
according to \cite{Dixon2003}, 
the probability that a non-trivial fixed word $w$ equals identity tends to $0$ as $n$ goes to infinity, and the same holds for the sequence of symmetric groups, according to \cite{Gamburd2009}. Thus, for every quantifier formula
$\phi(x_1,\dots,x_p)$, the probability that $\phi(x_1,\dots,x_p)$ will hold in $G_n$ for $n$ going to infinity for independent uniform random assignments of $x_1,\dots,x_p$ tends either to $1$ or $0$, depending on the fact that $\phi(x_1,\dots,x_p)$ holds trivially or not.
Also, according to \cite{Epstein1971}, the
the probability that a non-trivial fixed word $w$ equals identity is $0$ for 
a connected finite-dimensional nonsolvable Lie groups (for instance ${\rm GL}(2,\bbbr)$). Thus ${\rm GL}(2,\bbbr)$  is a QF-limit
of $(G_n)_{n\in\bbbn}$.
\end{example}

\subsection{Modelings}
We introduced in
\cite{NPOM1arxiv} --- as candidate for a possible limit object of sparse structures --- the notion of modeling,
which extends the notion of graphing introduced for bounded degree graphs.
A {\em relational sample space} is a relational structure $\mathbf A$ (with signature $\lambda$) 
with additional structure:
The domain $A$ of $\mathbf A$  is a standard Borel space
(with Borel $\sigma$-algebra $\Sigma_{\mathbf A}$)
 with the property that every subset of $A^p$ that is first-order definable 
 in ${\rm FO}(\lambda)$
  is measurable (in $A^p$ with respect to the product $\sigma$-algebra). 
	A {\em modeling} is a relational sample space equipped with a probability measure (denoted $\nu_{\mathbf A}$).
	For brevity we shall use the same letter $\mathbf A$  for  structure, relational sample space, and modeling.
The definition of modelings allows us to extend Stone pairing naturally to modelings:
the {\em Stone pairing} $\langle \phi,\mathbf A\rangle$ of a first-order formula $\phi$ 
(with free variables in $\{x_1,\dots,x_p\}$) and  
a modeling $\mathbf A$, 
is defined by 
$$
\langle \phi,\mathbf A\rangle=\nu^{\otimes p}(\phi(\mathbf A)),$$
where  
$$\phi(\mathbf A)=\{(v_1,\dots,v_p)\in A^p:\quad \mathbf A\models\phi(v_1,\dots,v_p)\}.$$
Note that every finite structure canonically defines a modeling (with same universe, discrete $\sigma$-algebra, and
uniform probability measure) and that in the definition above matches the definition of Stone pairing of a formula
and a finite structure introduced earlier.
Also note that every modeling $\sigma$-structure is obviously a Borel $\sigma$-structure, but that the converse does not have to hold in general, as projections of Borel sets are not Borel in general. 

For a fragment $X$ of ${\rm FO}$, we shall say that a sequence $(\mathbf A_n)_{n\in\bbbn}$ is {\em $X$-convergent} to a modeling $\mathbf L$ (or
that $\mathbf L$ is a {\em modeling $X$-limit} of $(\mathbf A_n)_{n\in\bbbn}$), and we shall note $\mathbf A_n\xrightarrow{X}\mathbf L$, if for every $\phi\in X$ it holds $\lim_{n\rightarrow\infty}\langle\phi,\mathbf{A}_n\rangle=\langle\phi,\mathbf{L}\rangle$.
We shall say that a class $\mathcal{C}$ of $\sig$-structures {\em has modeling $X$-limits} if,  every $X$-convergent sequence
$(\mathbf A_n)_{n\in\bbbn}$ of $\sig$-structures with $\mathbf A_n\in\mathcal C$ has a modeling $X$-limit. For instance,
we proved in \cite{NPOM1arxiv} that classes of graphs with bounded degrees and classes of graphs with bounded tree-depth have modeling ${\rm FO}$-limits, and we proved in \cite{NPOM2arxiv} that the class of rooted ($k$ vertex-colored) trees has modeling ${\rm FO}$-limits. However, the notion of modeling limit is, in some sense, only applicable to sparse structures:
\begin{theorem}[\cite{NPOM1arxiv}]
\label{thm:modnd}
Let $\mathcal C$ be a monotone class of graphs. 

If $\mathcal C$ has modeling ${\rm FO}$-limits then $\mathcal C$ is {\em nowhere dense}, that is: for every integer $p$ there exists an integer $N=N(\mathcal C,p)$ such that no graph in $\mathcal C$ contains the $p$ subdivision of $K_N$ as a subgraph.
\end{theorem}
For more on nowhere dense graphs, we refer the reader to
\cite{ND_logic,Nevsetvril2010a,ND_characterization,Sparsity}. The importance of nowhere dense classes
and the strong relationship of this notion with first-order logic is exemplified by the recent result of
Grohe, Kreutzer, and Siebertz \cite{Grohe2013}, which states that (under a reasonable complexity theoretic assumption)
deciding first-order properties of graphs in a monotone class $\mathcal C$ is fixed-parameter tractable if and only if
$\mathcal C$ is nowhere dense. 
Motivated by strong model theoretical properties and manifold characterizations of nowhere dense classes, we conjectured that the above theorem is tight.
This may be seen as the most challenging problem of this paper.
\begin{mainconjecture}{1}
Every nowhere dense class of graphs has modeling ${\rm FO}$-limits.
\end{mainconjecture}

The problem of the existence of a modeling limit can be reduced to 
the study of ${\rm FO}^{\rm local}$-convergence, and then to two ``typical'' particular cases:

A modeling $\mathbf A$ with universe $A$ satisfies the {\em Finitary Mass Transport Principle} if, 
for every $\phi,\psi\in{\rm FO}_1(\sig)$ and every integers $a,b$ such that
$$\begin{cases}
\phi\entails (\exists^{\geq a}y)\,(x_1\sim y)\wedge\psi(y)\\
\psi\entails (\exists^{\leq b}y)\,(x_1\sim y)\wedge\phi(y)
\end{cases}$$
it holds
$$a\,\langle\phi,\mathbf A\rangle\leq b\,\langle\psi,\mathbf A\rangle.$$
It is clear that every finite structure satisfies the Finitary Mass Transport Principle, hence
every modeling ${\rm FO}$-limit of finite structures satisfies the Finitary Mass Transport Principle, too.

The following stronger version of this principle, which is also satisfied by every finite structure, does not
automatically hold in the limit.

A modeling $\mathbf A$ with universe $A$ satisfies the {\em Strong Finitary Mass Transport Principle} if, 
for every measurable subsets $X,Y$ of $A$, and every integers $a,b$, the following property holds:
\begin{quote}
If every $x\in X$ has at least $a$ neighbors in $Y$ and every $y\in Y$ has at most $b$ neighbors in $X$ then
$a\,\nu_{\mathbf A}(X)\leq b\,\nu_{\mathbf A}(Y)$.
\end{quote}

The importance of residual and non-dispersive sequences appears again in the following result, which complements Theorem~\ref{thm:ecomb}.
\begin{theorem}[\cite{NPOM2arxiv}]
\label{thm:hmod}
Let $\mathcal C$ be a hereditary class of structures.

Assume that for every $\mathbf A_n\in\mathcal C$ and every $\rho_n\in A_n$  ($n\in\bbbn$) the following
properties hold:
\begin{enumerate}
	\item  if $(\mathbf A_n)_{n\in\bbbn}$ is  ${\rm FO}_1^{\rm local}$-convergent and residual, then it has
a modeling ${\rm FO}_1^{\rm local}$-limit;
\item  if $(\mathbf A_n,\rho_n)_{n\in\bbbn}$ is ${\rm FO}^{\rm local}$-convergent and
$\rho$-non-dispersive  then it has
a modeling ${\rm FO}^{\rm local}$-limit.
\end{enumerate}

Then $\mathcal C$ has modeling ${\rm FO}$-limits.

Moreover, if in cases (1) and (2) the modeling limits satisfy the Strong Finitary Mass Transport Principle, then
$\mathcal C$ has modeling ${\rm FO}$-limits that satisfy the Strong Finitary Mass Transport Principle.
\end{theorem}
\section{Constructions}
\label{sec:cons}
In this section we add two useful extensions of some of the above results, by means of two basic operations. The first is the notion of (linear) convex combination of modelings. This may be used to give an inverse theorem for the comb structure Theorem~\ref{thm:ecomb}. The second introduces the powerful tool of interpretation schemes, which gives the possibility to transport results from a class of $\sig$-structures to a class of $\sigma$-structures.
This also leads to a new characterization of nowhere dense classes.
\subsection{Convex Combinations}
\begin{definition}[Convex combination of Modelings]
Let $\mathbf H_i$ be $\sig$-modelings for $i\in I\subseteq\bbbn$ and let $(\alpha_i)_{i\in I}$ be positive real numbers
such that $\sum_{i\in I}\alpha_i=1$.

Let $\mathbf H$ be the disjoint union of the $\mathbf H_i$, let 
$\Sigma_{\mathbf H}=\{\bigcup_i X_i:\ X_i\in\Sigma_{\mathbf H_i}\}$ and, for
$X\in\Sigma_{\mathbf H}$,
let $\nu_{\mathbf H}(X)=\sum_i\alpha_i\nu_{\mathbf H_i}(X\cap H_i)$.

Then $\mathbf H$ is the {\em convex combination} of modelings
$\mathbf H_i$ with {\em weights} $\alpha_i$ and we denote it by 
$\coprod_{i\in I}(\mathbf H_i,\alpha_i)$.
\end{definition}

The following converse of Theorem~\ref{thm:ecomb}  is proved in~\cite{NPOM1arxiv}:
\begin{theorem}
\label{thm:convcomb}
Assume $J$ is a countable set, $\alpha_i$ ($i\in I$) are reals, and $(\mathbf B_{i,n})_{n\in\bbbn}$ ($i\in I$) are
sequences of $\sig$-structures such that
	$\alpha_i=\lim_{n\rightarrow\infty}\frac{|\mathbf B_{i,n}|}{|\bigcup_{j\in I}\mathbf B_{j,n}|}\ (\forall i\in I)$,
	$\sum_{i\in I}\alpha_i=1$, and for each $i\in I$, $(B_{i,n})_{n\in\bbbn}$ is ${\rm FO}^{\rm local}$-convergent.
Then $\mathbf A_n=\bigcup_{i\in I}\mathbf B_{i,n}$ is ${\rm FO}^{\rm local}$-convergent.
Also, if there exist  $\sig$-modelings $\mathbf L_i$ ($i\in I$) such that
 for each $i\in I$, $\mathbf B_{i,n}\xrightarrow{{\rm FO}^{\rm local}}\mathbf L_i$,
then $\mathbf A_n\xrightarrow{{\rm FO}^{\rm local}}\coprod_{i\in I}(\mathbf L_i,\alpha_i)$.

Moreover, if the family $\{(\mathbf B_{i,n})_{n\in\bbbn}: i\in I\}$ is uniformly elementarily-convergent,
then $(\mathbf A_n)_{n\in\bbbn}$ is ${\rm FO}$-convergent.
Also, if there exist  $\sig$-modelings $\mathbf L_i$ ($i\in I$) such that for each $i\in I$ it holds
$\mathbf B_{i,n}\xrightarrow{{\rm FO}}\mathbf L_i$ (for $i>0$ and $i=0$ if $\alpha_0>0$)
and
$\mathbf B_{0,n}\xrightarrow{{\rm FO}_0}\mathbf L_0$ (if $\alpha_0=0$)
then
$\mathbf A_n\xrightarrow{{\rm FO}}\coprod_{i\in I}(\mathbf L_i,\alpha_i)$.
\end{theorem}
\subsection{Interpretation Schemes}
Interpretation Schemes (introduced in this setting in~\cite{NPOM1arxiv}) generalize to other logics than first-order logic.

\begin{definition}
\label{def:interp}
Let $\mathcal L$ be a logic (here either first-order logic or ${\mathcal L}_{\omega_1\omega}$). For $p\in\bbbn$ and 
a signature $\sig$, $\mathcal L_p(\sig)$ denotes the set of the formulas in the language of $\sig$ in logic
$\mathcal L$, with free variables in $\{x_1,\dots,x_p\}$. 

Let $\kappa,\sig$ be signatures, where $\sig$ has 
$q$ relational symbols $R_1,\dots,R_q$ with respective arities $r_1,\dots,r_q$.

An {\em $\mathcal L$-interpretation scheme} ${\mathsf I}$ of $\sig$-structures
in $\kappa$-structures is defined by an integer $k$ --- the {\em exponent} of the $\mathcal L$-interpretation scheme --- a formula
$E\in{\mathcal L}_{2k}(\kappa)$, a formula $\theta_0\in{\mathcal L}_k(\kappa)$, and
 a formula 
$\theta_i\in{\rm FO}_{r_ik}(\kappa)$ for each symbol $R_i\in\sig$, such that:
\begin{itemize}
  \item the formula $E$ defines an equivalence relation
  of $k$-tuples;
   \item each formula $\theta_i$ is compatible with $E$, in the sense that
   for
   every $0\leq i\leq q$ it holds $$
  \bigwedge_{1\leq j\leq r_i}\,E(\mathbf x_j,\mathbf y_j)\quad\entails\quad
  \theta_i(\mathbf x_1,\dots,\mathbf x_{r_i})\leftrightarrow\theta_i(\mathbf
  y_1,\dots,\mathbf y_{r_i}),
   $$
   where $r_0=1$, boldface $\mathbf x_j$ and $\mathbf y_j$ represent
   $k$-tuples of free variables, and 
   where $\theta_i(\mathbf x_1,\dots,\mathbf x_{r_i})$ stands for
 $\theta_i(x_{1,1},\dots,x_{1,k},\dots,x_{r_i,1},\dots,x_{r_i,k})$.
\end{itemize}

For a $\kappa$-structure $\mathbf A$, we denote by $\mathsf{I}(\mathbf A)$ the
$\sig$-structure $\mathbf B$ defined as follows:
 \begin{itemize}
   \item the domain $B$ of $\mathbf B$ is the subset
   of the $E$-equivalence classes $[\mathbf x]\subseteq A^k$  of the tuples $\mathbf x=(x_1,\dots,x_k)$ 
  such that $\mathbf A\models \theta_0(\mathbf x)$;
   \item for each $1\leq i\leq q$ and every 
   $\mathbf v_1,\dots,\mathbf v_{s_i}\in A^{kr_i}$ such that
   $\mathbf A\models\theta_0(\mathbf v_j)$ (for every $1\leq j\leq r_i$) it
   holds
   $$
   \mathbf B\models R_i([\mathbf v_1],\dots,[\mathbf v_{r_i}])\quad\iff\quad
   \mathbf A\models \theta_i(\mathbf v_1,\dots,\mathbf v_{r_i}).
   $$  
 \end{itemize}
\end{definition}

From the standard properties of model theoretical interpretations
(see, for instance \cite{Lascar2009} p.~180), we state the following: if
$\mathsf I$ is an $\mathcal L$-interpretation of $\sig$-structures in $\kappa$-structures,
then there exists a mapping
$\tilde{\mathsf I}:{\mathcal L}(\sig)\rightarrow {\mathcal L}(\kappa)$ (defined by
means of the formulas $E,\theta_0,\dots,\theta_q$ above) such that
for every $\phi\in {\mathcal L}_p(\sig)$, and every $\kappa$-structure $\mathbf A$,
 the following property holds (while letting $\mathbf B=\mathsf I(\mathbf A)$
 and identifying elements of $B$ with the corresponding equivalence classes of $A^k$):

For every $[\mathbf v_1],\dots,[\mathbf v_p]\in B^{p}$ (where $\mathbf v_i=(v_{i,1},\dots,v_{i,k})\in A^k$)
it holds
$$
  \mathbf B\models \phi([\mathbf v_1],\dots,[\mathbf v_p])\quad\iff\quad \mathbf
  A\models \tilde{\mathsf I}(\phi)(\mathbf v_1,\dots,\mathbf v_p).$$
	
It directly follows from the existence of the mapping $\tilde{\mathsf I}$ that
\begin{itemize}
	\item an ${\rm FO}$-interpretation scheme ${\mathsf I}$ of $\sig$-structures in
$\kappa$-structures defines a continuous mapping from $S(\mathcal B({\rm
FO}(\kappa)))$ to $S(\mathcal B({\rm FO}(\sig)))$;
	\item an $\mathcal L_{\omega_1\omega}$-interpretation scheme ${\mathsf I}$ of $\sig$-structures in
$\kappa$-structures defines a measurable mapping from $S(\mathcal B({\rm
FO}(\kappa)))$ to $S(\mathcal B({\rm FO}(\sig)))$.
\end{itemize}

\begin{definition}
Let $\kappa,\sig$ be signatures.
A {\em basic $\mathcal L$-interpretation scheme} ${\mathsf I}$ of $\sig$-structures
in $\kappa$-structures  with {\em exponent} $k$ is defined by a formula 
$\theta_i\in{\mathcal L}_{kr_i}(\kappa)$ for each symbol $R_i\in\sig$ with arity
$r_i$. 

For a $\kappa$-structure $\mathbf A$, we denote by $\mathsf{I}(\mathbf A)$ the
structure with domain $A^k$ such that, for every $R_i\in\sig$ with arity $r_i$ and every
$\mathbf v_1,\dots,\mathbf v_{r_i}\in A^k$ it holds
$$
\mathsf I(\mathbf A)\models R_i(\mathbf v_1,\dots,\mathbf v_{r_i})\quad\iff\quad
\mathbf A\models\theta_i(\mathbf v_1,\dots,\mathbf v_{r_i}).
$$
\end{definition}

It is immediate that every basic $\mathcal L$-interpretation scheme $\mathsf I$ defines
a mapping $\tilde{\mathsf I}:\mathcal L(\sig)\rightarrow \mathcal L(\kappa)$ such
that for every $\kappa$-structure $\mathbf A$, every $\phi\in\mathcal L_p(\sig)$,
and every $\mathbf v_1,\dots,\mathbf v_p\in A^k$ it holds
$$
\mathsf I(\mathbf A)\models \phi(\mathbf v_1,\dots,\mathbf v_{p})\quad\iff\quad
\mathbf A\models\tilde{\mathsf I}(\phi)(\mathbf v_1,\dots,\mathbf v_{p})
$$

We deduce the following general properties:

\begin{proposition}[\cite{NPOM1arxiv}]
\label{lemma:interpFO}
Let $\mathsf I$ be an ${\rm FO}$-interpretation scheme of $\sig$-structures in $\kappa$-structures.

Then, if a sequence $(\mathbf A_n)_{n\in\bbbn}$ of finite $\kappa$-structures
is ${\rm FO}$-convergent then the sequence 
$(\mathsf{I}(\mathbf A_n))_{n\in\bbbn}$ of (finite) $\sig$-structures
is ${\rm FO}$-convergent. 
\end{proposition}

\begin{proposition}[\cite{NPOM2arxiv}]
\label{lemma:interpomega}
Let $\mathsf I$ be a basic  ${\mathcal L}_{\omega_1\omega}$-interpretation scheme 
of $\sig$-structures in $\kappa$-structures.

For every $\kappa$-modeling $\mathbf A$, 
$\mathsf I(\mathbf A)$ is a $\sig$-modeling such that  for every $\phi\in\mathcal L_p(\sig)$ it holds
$$\langle\phi,\mathsf I(\mathbf A)\rangle=\langle\tilde{\mathsf I}(\phi),\mathbf A\rangle.$$
\end{proposition}

Interpretations are a powerful tool. For a recent application to graph polynomials, see \cite{trivial_arxiv}.
\section{Open Problems}
\label{sec:pb}

Our main conjecture here is Conjecture~\ref{conj:ND}, which asserts that every monotone nowhere dense class of graphs has modeling limits.

There is a modest progress toward a resolution of this conjecture:
\begin{itemize}
\item any monotone class of graphs with bounded degree has modeling limit \cite{CMUC},
\item the class of (colored rooted) forests has modeling limits \cite{NPOM2arxiv}. 
\end{itemize}
As an extension of the results in \cite{NPOM2arxiv}, 
every ${\rm FO}$-convergent sequence $(G_n)_{n\in\bbbn}$ with 
${\rm girth}(G_n)\rightarrow\infty$ has a modeling limit.
From this follows the first positive result concerning nowhere dense classes which do not have bounded average degree.

For a graph $G$, we denote by $b_1(G)$ the {\em cyclomatic number} of  $G$ (i.e. $b_1(G)=\|G\|-|G|+1$), and by ${\rm Diam}(G)$ its diameter.
\begin{proposition}
\label{prop:ND}
Let $f:\bbbn\rightarrow\bbbn$ be a mapping and let 
$\mathcal C_f$ be the monotone nowhere dense class of all (colored) graphs $G$ such that for every graph $G\in\mathcal C$  it holds
$$b_1(G)\leq f({\rm Diam}(G)).$$
Then the class $\mathcal C_f$ has modeling limits.
\end{proposition}
Note that such classes include in particular all classes with bounded degree, and more generally classes of graphs in which the maximum degree is bounded by a function of the girth.
Proposition~\ref{prop:ND} is currently our most general result toward Conjecture~\ref{conj:ND}.

\begin{problem}
\label{pb:colin}
We consider $2$-colored linear order, that is transitive tournaments
with vertices colored black or white.

Do every ${\rm FO}$-convergent sequence of $2$-colored linear order admit a modeling ${\rm FO}$-limit?
\end{problem}

Note that if the answer to Problem~\ref{pb:colin} is yes, then it follows (by suitable interpretation) that every class of graphs with bounded pathwidth has modeling ${\rm FO}$-limits.
Similarly, if the above Example~\ref{ex:tsl} extends to ${\rm FO}$-convergence in the sense that every ${\rm FO}$-convergent sequence of tree semilattices with $k$-colored domain has a
modeling ${\rm FO}$-limit then it would imply that every 
class of graphs with bounded treewidth has modeling ${\rm FO}$-limits.

Quantifier-free fragment corresponds to the best known (and thoroughly investigated) type of convergence, which is left-convergence in the case of graphs and hypergraphs. However, our knowledge of QF-convergence is far from being complete. First, it is not completely clear what form the limit object should have in case of a relational structures with several non-unary relations. More important, the extension of the theory (notions of (hyper)graphons, cut metric, regularity theorem) in the simple case of a signature reduced to a single binary function is completely open. In general this problem might be too difficult to be directly addressed. It appears that a natural restriction consists in limiting the power of the binary function:
\begin{problem}
Let $\sigma$ be a signature reduced to a single binary function $*$ (denoted as an operation), and let $T_0$ be a finite set of universal formulas (in the first-order language of $\sigma$).
 Assume that there exists a function $f$ such that for every $p$, every model $\mathbf A$ of $T_0$, and every $v_1,\dots,v_p\in A$, the $\sigma$-structure 
$\mathbf{A}\langle v_1,\dots,v_p\rangle$ generated by $v_1,\dots,v_p$ --- that is the substructure of $\mathbf A$ whose domain is the smallest subset of $A$ which contains $v_1,\dots,v_p$ and is
closed under the operation $*$ --- has a domain of cardinality at most $f(p)$.

Does there exist a graphon-like limit object for QF-convergent sequences of finite models of $T_0$? Does there exist a regularity theorem for finite models of $T_0$?
\end{problem}

Note that in some simple cases, we can answer positively. For instance in the case of tree-semilattices \cite{QFlim}, which corresponds to associative, commutative, idempotent operation $*$ with the additional property
$$
\forall x\,\forall y\,\forall z\quad [(x*z=x)\wedge(y*z=y)]\rightarrow
[(x*y=x)\vee(x*y=y)].
$$
However, this case is too simple in the sense that limit objects are
Borel tree-semilattices, which means that tree-semilattices form a 
``random-free'' class.

A possible candidate for a tractable non random-free example is the class of general semilattices, defined by general associative, commutative, idempotent operations $*$.

We can also consider signatures $\sigma$ only containing unary functional symbols.
In the case of a signature $\sigma$ reduced to a single unary function, it appears that not only 
a limits of QF-convergent sequences of finite structures can be
represented by a Borel $\sigma$-structure, but the class of these $\sigma$-structures actually admit modeling ${\rm FO}$-limits, as an extension of the results in \cite{NPOM2arxiv}.
Note that for a signature $\sigma$ consisting of at least two unary functional symbols, 
there is no hope to have modeling ${\rm FO}$-limits (as such a class is sufficiently general to allow in it an interpretation of the class of all directed graphs)
\providecommand{\noopsort}[1]{}\providecommand{\noopsort}[1]{}
\providecommand{\bysame}{\leavevmode\hbox to3em{\hrulefill}\thinspace}
\providecommand{\MR}{\relax\ifhmode\unskip\space\fi MR }
\providecommand{\MRhref}[2]{%
  \href{http://www.ams.org/mathscinet-getitem?mr=#1}{#2}
}
\providecommand{\href}[2]{#2}

\end{document}